\documentclass[a4paper]{scrartcl}
\usepackage[utf8]{inputenc}
\usepackage[english]{babel}
\usepackage{xcolor}
\usepackage{graphicx}
\usepackage{hyperref}
\usepackage[T1]{fontenc}
\usepackage{lmodern}
\usepackage{rotating}
\usepackage{tikz}
\usepackage{tikz-cd}
\usepackage[backend=biber,sortlocale=de_DE,url=false,doi=false,isbn=false,giveninits=true,style=alphabetic,maxbibnames=1000,date=year,sorting=nyt]{biblatex}
\bibliography{references.bib}
\DeclareFieldFormat*{title}{\mkbibemph{#1\isdot}}
\DeclareFieldFormat*{journaltitle}{#1\isdot}
\renewbibmacro{in:}{}
\usepackage{amsmath}
\usepackage{amsfonts}
\usepackage{amssymb}
\usepackage{amsthm}
\usepackage{mathtools}
\usepackage{stmaryrd}
\usepackage{enumerate}
\usepackage{verbatim}
\usepackage{dsfont}
\usepackage{mathabx}
\usepackage{faktor}

\DeclareMathOperator\Gal{Gal}

\DeclareMathOperator\supp{supp}

\hypersetup{colorlinks=false, linkbordercolor={white}, hidelinks, pdfauthor={Felix Schremmer},pdftitle={Newton strata in Levi subgroups}}
\usepackage{csquotes}
\author{Felix Schremmer}\date{\today}
\title{Newton strata in Levi subgroups}
\newtheorem{theorem}[equation]{Theorem}
\newtheorem{proposition}[equation]{Proposition}

\newtheorem{corollary}[equation]{Corollary}

\theoremstyle{definition}
\newtheorem{definition}[equation]{Definition}

\theoremstyle{remark}

\let\oldqedsymbol\qedsymbol
\def\qedaddendum{}
\def\qedsymbol{\oldqedsymbol\qedaddendum}

\def\rightqed{\pushQED{\qed}\qedhere\popQED}

\begin{document}
\maketitle

\begin{abstract}
Certain Iwahori double cosets in the loop group of a reductive group, known under the names of $P$-alcoves or $(J,w,\delta)$-alcoves, play an important role in the study of affine Deligne-Lusztig varieties. For such an Iwahori double coset, its Newton stratification is related to the Newton stratification of an Iwahori double coset in a Levi subgroup. We show that this relation is closer than previously known.
\end{abstract}

Let $G$ be a reductive group over the local field $F$, whose completion of the maximal unramified extension we denote by $\breve F$. In the context of Langlands program, one would choose $F$ to be a finite extension of $p$-adic rationals, whereas in the context of moduli spaces of shtukas, $F$ would be the field of formal Laurent series over a finite field. In any case, the Galois group $\Gamma = \Gal(\breve F/F)$ is generated by the Frobenius $\sigma$. We moreover pick a $\sigma$-stable Iwahori subgroup $I\subseteq G(\breve F)$. Then the affine Deligne-Lusztig variety associated to two elements $x,b\in G(\breve F)$ is defined as
\begin{align*}
X_x(b) = \{g\in G(\breve F)/I\mid g^{-1}b\sigma(g)\in IxI\}.
\end{align*}

Evidently, the isomorphism type of $X_x(b)$ only depends on the $\sigma$-conjugacy class
\begin{align*}
[b] = \{g^{-1}b\sigma(g)\mid g\in G(\breve F)\}\subseteq G(\breve F)
\end{align*}
and the Iwahori double coset $IxI$. The latter Iwahori double cosets are typically indexed using the extended affine Weyl group $\widetilde W$, so that each coset is given by $I\dot x I$ for a uniquely determined element $x\in \widetilde W$. The set $B(G)$ of $\sigma$-conjugacy classes has an important parametrization due to Kottwitz \cite{Kottwitz1985, Kottwitz1997}, characterizing each $[b]\in B(G)$ by its Newton point $\nu(b)$ and Kottwitz point $\kappa(b)$.

Following \cite[Section~2]{Goertz2015}, we assume without loss of generality that the group $G$ is quasi-split over $F$. We choose a maximal torus $T$ whose unique parahoric subgroup $T_0(\breve F)$ is contained in $I$ and a $\sigma$-stable Borel subgroup $T\subset B$ such that, in the corresponding appartment of the Bruhat-Tits building of $G_{\breve F}$, the alcove fixed by $I$ is opposite to the dominant cone defined by $B$. With this notation, the Kottwitz point $\kappa(b)$ for $b\in G(\breve F)$ lies in $(X_\ast(T)/\mathbb Z\Phi^\vee)_{\Gamma}$, where $\Phi^\vee$ is the set of coroots. The dominant Newton point $\nu(b)$ is an element of $X_\ast(T)_{\Gamma_0}\otimes\mathbb Q$, where $\Gamma_0\subset \Gamma$ is the absolute Galois group of $\breve F$. We identify the extended affine Weyl group $\widetilde W$ as the semidirect product of the finite Weyl group $W = N_G(T)/T$ with $X_\ast(T)_{\Gamma_0}$.

Geometric properties of the affine Deligne-Lusztig variety $X_x(b)$ are closely related to those of the corresponding Newton stratum $[b]\cap IxI\subset G(\breve F)$. Obviously, one is empty if and only if the other is empty. Further properties can be related following \cite[Section~3.1]{Milicevic2020}.

It is an important question to study which of these affine Deligne-Lusztig varieties are non-empty, i.e.\ to determine the set \begin{align*}
B(G)_x = \{[b]\in B(G)\mid [b]\cap IxI\neq\emptyset\}.
\end{align*} An important breakthrough result of Görtz-He-Nie \cite{Goertz2015} is a characterization of all elements $x\in \widetilde W$ where $B(G)_x$ contains the basic $\sigma$-conjugacy class. For this purpose, they introduce the notion of a $(J,w,\sigma)$-alcove, generalizing the previous notion of a $P$-alcove known for split groups. Write  $\Delta\subseteq \Phi$ for the set of simple roots.
\begin{definition}
Let $x\in\widetilde W, w\in W$ and $J\subseteq \Delta$ such that $J=\sigma(J)$. Then we say that $x$ is a \emph{$(J,w,\sigma)$-alcove element} if the following conditions are both satisfied:
\begin{enumerate}[(a)]
\item The element $\tilde x = w^{-1} x\sigma(w)$ lies in the extended affine Weyl group $\widetilde W_M$ of the standard Levi subgroup $M=M_J\supseteq T$ defined by $J$.
\item For all positive roots $\alpha\in \Phi^+$ that are not in the root system $\Phi_J$ generated by $J$, the corresponding root subgroup $U_\alpha\subseteq G(\breve F)$ satisfies
\begin{align*}
U_{w\alpha}\cap \prescript x{}{}I\subseteq U_{w\alpha}\cap I.
\end{align*}
\end{enumerate}
\end{definition}
Observe that $x$ is a $(J,w,\sigma)$-alcove element if it is a $(J,w',\sigma)$-alcove element for any $w'\in wW_J$, where $W_J$ is the subgroup of $W$ generated by the simple reflections coming from $J$ (or equivalently the finite Weyl group of $M_J$). Following Viehmann \cite[Section~4]{Viehmann2021}, we say that $x$ is a \emph{normalized} $(J,w,\sigma)$-alcove element if $w$ has minimal length in its coset $wW_J$.

Assume now that $x$ is a normalized $(J,w,\sigma)$-alcove element and write $\tilde x = w^{-1} x\sigma(w)\in\widetilde W_M$. If $[b]\in B(G)_x$, it is a result of Görtz-He-Nie \cite[Theorem~3.3.1]{Goertz2015} that each element in Newton stratum $IxI\cap [b]$ is of the form $i^{-1} w m \sigma(w^{-1} i)$ for some $i\in I$ and
\begin{align*}
m\in \bigcup_{[b']} \Bigl((I\cap M)\tilde x(I\cap M)\cap [b']\Bigr)\subseteq M,
\end{align*}
with the union taken over all $[b']\in B(M)$ contained in $[b]$. Our main result states that this union is spurious.
\begin{theorem}\label{thm:sigmaConjugacyBijection}
Let $x$ be a normalized $(J,w,\sigma)$-alcove element and $\tilde x = w^{-1}x\sigma(w)\in \widetilde W_M$.
Then we get a bijective map
\begin{align*}
B(M)_{\tilde x}\rightarrow B(G)_x,
\end{align*}
sending a $\sigma$-conjugacy class $[b]_M\in B(M)_{\tilde x}$ to the unique $\sigma$-conjugacy class $[b]_G\in B(G)$ with $[b]_M\subseteq [b]_G$. In this case, the dominant Newton points $\nu_M(b)$ and $\nu_G(b)$ agree as elements of $X_\ast(T)_{\Gamma_0}\otimes\mathbb Q$.
\end{theorem}
Unfortunately, the relationship discussed here has been a source of confusion in the past. While surjectivity of the map $B(M)_{\tilde x}\rightarrow B(G)_x$ follows from \cite[Theorem~3.3.1]{Goertz2015}, it should be noted that the natural map $B(M)\rightarrow B(G)$ is neither injective nor surjective. In particular, for $[b]_G\in B(G)_x$, the intersection $[b]_G\cap M$ may consist of more than one $\sigma$-conjugacy class of $M$. This point is further discussed in the erratum\footnote{Cf.\ \url{https://www.esaga.uni-due.de/f/ulrich.goertz/pdf/Erratum-GHN.pdf} .} to \cite{Goertz2015} and in \cite[Example~5.4]{Viehmann2021}. This latter example moreover demonstrates that the assumption of a \emph{normalized} $(J,w,\delta)$-alcove element is essential for Theorem~\ref{thm:sigmaConjugacyBijection} to hold.

Our theorem together with \cite[Theorem~3.3.1]{Goertz2015} yields a canonical isomorphism $X_{\tilde x}(b)\rightarrow X_x(b)$, sending $g (I\cap M) \in X_{\tilde x}(b)$ to $gw^{-1} I\in X_x(b)$.
The geometric correspondence is mirrored by a corresponding representation-theoretic result of He-Nie \cite[Theorem~C]{He2015b}, comparing class polynomials of $x$ and $\tilde x$. These are certain structure constants describing the cocenter of the Iwahori-Hecke algebra of $\widetilde W$ resp.\ $\widetilde W_M$. If one knows all class polynomials for a given element $x\in\widetilde W$, one can use these to determine many geometric properties the affine Deligne-Lusztig varieties $X_x(b)$ for $[b]\in B(G)$, cf.\ \cite[Theorem~6.1]{He2014}. These properties include dimension as well as the number of top dimensional irreducible components up to the action of the $\sigma$-centralizer of $b\in G(L)$. Moreover, in a certain sense, the number of rational points of the Newton stratum $IxI\cap [b]$ can be expressed using these class polynomials \cite[Proposition~3.7]{He2022}. In this sense, it already follows from \cite{He2015b} that these numerical invariants agree for $X_{\tilde x}(b)$ and $X_x(b)$.

\begin{proof}[Proof of Theorem~\ref{thm:sigmaConjugacyBijection}]
We follow the Deligne-Lusztig reduction method due to Görtz-He \cite{Goertz2010b}, i.e.\ we do an induction on $\ell(x)$.

If $x$ is of minimal length in its $\sigma$-conjugacy class, then $B(G)_x$ contains only one element, being the $\sigma$-conjugacy class defined by $x$. Moreover, He-Nie \cite[Proposition~4.5]{He2015b} prove that in this case $\tilde x$ is of minimal length in its $\sigma$-conjugacy class in $\widetilde W_M$. Hence we only have to show that $\nu_M(\tilde x)$ agrees with $\nu_G(x)$.

Following the definition of Newton points, we consider $\sigma$-twisted powers
\begin{align*}
x^{\sigma,n} =x \sigma(x)\cdots \sigma^{n-1}(x)\in\widetilde W.
\end{align*}
Observe that each $x^{\sigma,n}$ is a $(J,w,\sigma^n)$-alcove element. Let $n$ be sufficiently large such that $x^{\sigma,n}$ is a pure translation element, i.e.\ equal to the image of some $\mu\in X_\ast(T)_{\Gamma_0}$ in $\widetilde W$, and such that $\sigma^n$ is the identity map on $\widetilde W$. Then the Newton point $\nu_G(x)$ is the unique dominant element in the $W$-orbit of $\mu/n$. Similarly, the Newton point $\nu_M(\tilde x)$ is the unique dominant (with respect to $B\cap M$) element in the $W_J$-orbit of $w^{-1}\mu/n$.

The fact that $x^{\sigma,n}$ is a $(J,w,1)$-alcove element implies that $\langle w^{-1}\mu,\alpha\rangle\geq 0$ for all $\alpha\in \Phi^+\setminus\Phi_J$. Hence $\nu_M(\tilde x)$ is already dominant with respect to $B$, and thus $\nu_G(x) = \nu_M(\tilde x)$. This finishes the proof in case $x$ has minimal length in its $\sigma$-conjugacy class.

If $x$ is not of minimal length in its $\sigma$-conjugacy class, we can use \cite[Theorem~A]{He2014b} to obtain a sequence
\begin{align*}
x = x_1\xrightarrow{s_1}\cdots \xrightarrow{s_n} x_{n+1},
\end{align*}
for simple affine reflections $s_i\in \widetilde W$ and elements $x_{i+1} = s_i x_i \sigma(s_i)$ such that $\ell(x_1)=\cdots=\ell(x_{n}) >\ell(x_{n+1})$. From \cite[Lemma~7.1]{He2015b}, we find elements $w_1, \dotsc, w_n\in W$ such that each $x_i$ is a normalized $(J,w_i,\sigma)$-alcove element. Denote the corresponding elements by $\tilde x_i = w_i^{-1} x_i\sigma(w_i)\in\widetilde W_M$, so that the proof of \cite[Corollary~4.4]{He2015b} shows
\begin{align*}
\ell_{\widetilde W_M}(\tilde x_1) = \cdots=\ell_{\widetilde W_M}(\tilde x_{n})>\ell_{\widetilde W_M}(\tilde x_{n+1}).
\end{align*}
Said proof moreover reveals that each $\tilde x_{i+1}$ is conjugate to $\tilde x_i$ either by a simple affine reflection in $\widetilde W_M$ or a length zero element in $\widetilde W_M$.

The Deligne-Lusztig reduction method of Görtz-He \cite{Goertz2010b} yields
\begin{align*}
B(G)_x = &\cdots = B(G)_{x_{n}} = B(G)_{x_{n+1}}\cup B(G)_{s_n x_{n}},\\
B(M)_{\tilde x}=&\cdots = B(M)_{\tilde x_{n}}.
\end{align*}
We moreover know from the aforementioned article of He-Nie that $\tilde x_{n+1} = \tilde s \tilde x_{n}\sigma(\tilde s)$ for some simple affine reflection $\tilde s = w_n s_n w_n^{-1}\in \widetilde W_M$ of $M$. Hence
\begin{align*}
B(M)_{\tilde x_{n}} = B(M)_{\tilde x_{n+1}}\cup B(M)_{\tilde s\tilde x_{n}}.
\end{align*}
By induction, we get bijective and Newton-point preserving maps
\begin{align*}
B(M)_{\tilde x_{n+1}}\rightarrow B(G)_{x_{n+1}},\qquad B(M)_{\tilde s \tilde x_{n}}\rightarrow B(G)_{s_nx_{n}}.
\end{align*}
We conclude that the map $B(M)_{\tilde x}\rightarrow B(G)_x$ is well-defined, surjective and Newton-point preserving. If $[b_1]_M, [b_2]_M\in B(M)_{\tilde x}$ have the same image $[b]_G\in B(G)_x$ under this map, then $\nu_M(b_1) = \nu_G(b) = \nu_M(b_2)$ and $\kappa_M(b_1) = \kappa_M(\tilde x) = \kappa_M(b_2)$, hence $[b_1]_M = [b_2]_M$. This finishes the induction and the proof.
\end{proof}
Let us note the following consequence of Theorem~\ref{thm:sigmaConjugacyBijection}.
\begin{corollary}\label{cor:JwCongruence}
If $x$ is a $(J,w,\sigma)$-alcove element and $[b_1], [b_2]\in B(G)_x$, then \begin{align*}&\nu_G(b_1)\equiv \nu_G(b_2)\pmod{\Phi_J^\vee}.\rightqed\end{align*}
\end{corollary}
As an application of this corollary, we prove a conjecture of Dong-Gyu Lim \cite{Lim2023}, yielding an alternative criterion to the one from \cite{Goertz2015} for the non-emptiness of the basic Newton stratum in $IxI$.
\begin{definition}
Let $x\in\widetilde W$ be written as $x = \omega s_1\cdots s_{\ell(x)}$ for a length zero element $\omega$ and simple affine reflections $s_1,\dotsc,s_{\ell(x)}\in \widetilde W$. We define the \emph{$\sigma$-support} of $x$ to be the smallest subset $J\subseteq \widetilde W$ containing $s_1,\dotsc,s_{\ell(x)}$ and being closed under the action of the composite automorphism $\sigma\circ\omega$. Denote it by $\supp_\sigma(x)$. We say that $x$ is \emph{spherically $\sigma$-supported} if the subgroup of $\widetilde W$ generated by $\supp_\sigma(x)$ is finite.
\end{definition}
It follows from \cite[Proposition~5.6]{Goertz2019} that $x$ has spherical $\sigma$-support if and only if $B(G)_x = \{[b]\}$ for a basic $\sigma$-conjugacy class $[b]$.
\begin{proposition}Assume that the Dynkin diagram of $\Phi$ is $\sigma$-connected, i.e.\ that the Frobenius $\sigma$ acts transitively on the irreducible components of the root system $\Phi$.

Let $x\in\widetilde W$, and denote by $[b]\in B(G)$ the unique basic $\sigma$-conjugacy class with $\kappa(b) = \kappa(x)$. Then $X_x(b)=\emptyset$ if and only if the following two conditions are both satisfied:
\begin{enumerate}[(a)]
\item The element $x$ does \emph{not} have spherical $\sigma$-support, i.e.\ $B(G)_x$ contains a non-basic $\sigma$-conjugacy class.
\item There exists $J\subsetneq \Delta$ and $w\in W$ such that $x$ is a $(J,w,\sigma)$-alcove element.
\end{enumerate}
\end{proposition}
\begin{proof}
If $x$ has spherical $\sigma$-support, we get $IxI\subseteq [b]$, so that indeed $X_x(b)\neq\emptyset$. In the case that $x$ is not a $(J,w,\sigma)$-alcove element for any $J\subsetneq \Delta$, we easily obtain $X_x(b)\neq\emptyset$ by \cite[Theorem~A]{Goertz2015}.

Assume now conversely that (a) and (b) both hold true, so we have to show $X_x(b)=\emptyset$. Let $(J,w)$ be as in (b) such that moreover $x$ is a normalized $(J,w,\sigma)$-alcove element. Let $[b_x]\in B(G)_x$ denote the generic $\sigma$-conjugacy class.

Assume that $[b]\in B(G)_x$. From (b) together with Corollary~\ref{cor:JwCongruence}, we see \begin{align*}\nu(b)\equiv \nu(b_x)\pmod{\Phi_J^\vee}.\end{align*}
In particular
\begin{align*}
\langle \nu_G(b_x)-\nu_G(b),2\rho-2\rho_J\rangle=0.
\end{align*}
Since $\nu_G(b_x)$ is dominant and $b$ is basic, we conclude $\langle \nu_G(b_x),\alpha\rangle=0$ for all $\alpha\in \Phi^+\setminus\Phi_J$.
Thus $\Phi = \Phi_J\cup\Phi_{J'}$ where $J'\subseteq\Delta$ is the stabilizer of $\nu(b_x)$. By $\sigma$-irreducibility and looking at longest roots of irreducible components, we get $J=\Delta$ or $J'=\Delta$.

We assumed $J\neq\Delta$ in (b), so we conclude that $\nu_G(b_x)$ must be central. Thus $[b_x]$ is basic itself. This contradicts (a).
\end{proof}
\emph{Acknowledgements.} We thank Dong-Gyu Lim for explaining his conjecture and Eva Viehmann,  Xuhua He and Quingchao Yu for helpful discussions, in particular X.\ He for one pointing out the paper \cite{He2015b} and Q.\ Yu for pointing out \cite{Goertz2019}. We thank Sian Nie and Xuhua He for their comments on a preliminary version of this article.

\printbibliography
\end{document}